\documentclass[12pt,leqno]{amsart}
\usepackage{amsmath}
\usepackage{amssymb}
\usepackage{graphics}
\theoremstyle{plain}
\newtheorem{lemma}{Lemma}[section]
\newtheorem{theorem}[lemma]{Theorem}

\newtheorem{definition}[lemma]{Definition}

\renewcommand{\dim}{{\mathrm{dim}}}

\renewcommand{\dim}{{\mathrm{dim}}}
\newcommand{\C}{{\mathbb{C}}}
\newcommand{\A}{{\mathbb{A}}}

\renewcommand{\P}{{\mathbb{P}}}

\begin{document}
\date{} 
\title[Degenerate Second Main Theorems for Holomorphic Curves]{Degenerate Second Main Theorems for Holomorphic Curves in Different Geometric Settings}

\author{Si Duc Quang$^{1,2}$}
\author{Nguyen Van An$^3$}
\author{Tran An Hai$^3$}
\address{$^1$Department of Mathematics, School of Mathematics and Computation Science, Hanoi National University of Education, 136-Xuan Thuy, Cau Giay, Hanoi, Vietnam}
\address{$^2$Institute of Natural Sciences, Hanoi National University of Education, 136-Xuan Thuy, Cau Giay, Hanoi, Vietnam}
\address{$^3$Division of Mathematics, Banking Academy, 12-Chua Boc, Dong Da, Hanoi, Vietnam}
\email{quangsd@hnue.edu.vn; an0883@gmail.com; haita@hvnh.edu.vn}

 
\maketitle      
\begin{abstract}
We establish second main theorems for holomorphic curves into a projective subvary $V \subset \mathbb{P}^n(\mathbb{C})$ of dimension $k$, intersecting hypersurfaces in $N$-subgeneral position with respect to $V$ $(N > k)$. Our results provide explicit truncation levels for the counting functions that are independent of the number of hypersurfaces. The theorems are obtained in several settings, including holomorphic curves on $\mathbb{C}$, annuli, complex discs with finite growth index, and K\"ahler manifolds. We obtain a total defect bound that improves upon the previously known results. As an application, we establish a corresponding form of Schmidt's subspace theorem for families of homogeneous polynomials in subgeneral position.
\end{abstract}
\tableofcontents

\section{Introduction and Main results} 

The value distribution theory for meromorphic functions on $\C$ was initiated by R. Nevanlinna \cite{N} in 1926 and subsequently extended to higher dimensions. In 1933, H. Cartan \cite{C} established a second main theorem for linearly nondegenerate meromorphic mappings from $\C^m$ into $\P^n(\C)$ with respect to hyperplanes in general position. 

Using Nochka weights, E.  I. Nochka \cite{Noc83} generalized Cartan's theorem to hyperplanes in subgeneral position. More precisely, he proved the following.

\vskip0.2cm
\noindent
\textbf{Theorem A} (cf. \cite{Noc83,No05}). {\it Let $f: \C^m \to \P^n(\C)$ be a linearly nondegenerate meromorphic mapping and let $\{H_i\}_{i=1}^q$ be hyperplanes in $\P^n(\C)$ in $N$-subgeneral position. Then
$$
\|\ (q - 2N + n - 1)T_f(r) \leq \sum_{i=1}^q N^{[n]}(r,\nu_{(H_i,f)}) + o(T_f(r)).
$$}

Here $\|\ P$ means that $P$ holds for all $r \in [1,+\infty)$ outside a set of finite linear measure. The functions $T_f(r)$ and $N^{[n]}(r,\nu_{(H_i,f)})$ denote the characteristic function and the counting function truncated to level $n$ of $f^*H_i$, respectively. The family $\{H_i\}_{i=1}^q$ is said to be in $N$-subgeneral position if
$$\bigcap_{j=0}^{N} H_{i_j} = \varnothing \quad \text{for all } 1 \le i_0 < \cdots < i_N \le q.$$

In 1979, B. Shiffman \cite{S} conjectured that the total defect of an algebraically nondegenerate holomorphic curve $f: \C \to \P^n(\C)$ with respect to hypersurfaces in general position is bounded by $n+1$. This was proved by M. Ru \cite{MR} in 2004 via a second main theorem with optimal defect bound. M. Ru \cite{Ru09} later extended this to hypersurfaces in general position with respect to a projective subvariety. Using the hypersurface replacing technique, S. D. Quang \cite{Q19} further treated the subgeneral position case.

We recall the following definition.

\vskip0.2cm
\noindent
\textbf{Definition C}. {\it Let $V \subset \P^n(\C)$ be a projective subvariety of dimension $k$, and let $\mathcal Q=\{Q_1,\ldots,Q_q\}$ be hypersurfaces in $\P^n(\C)$. The family $\mathcal Q$ is in $N$-subgeneral position with respect to $V$ if
$$\bigcap_{j=0}^{N} (Q_{i_j} \cap V) = \varnothing \quad \text{for all } 1 \le i_0 < \cdots < i_N \le q.$$
If $N=k$, then $\mathcal Q$ is in general position with respect to $V$.}

The result of S. D. Quang \cite{Q19} is stated as follows.
 
\vskip0.2cm
\noindent
\textbf{Theorem D} (cf. \cite{Q19}) {\it  Let $V\subset\P^n(\C)$ be a smooth complex projective variety of dimension $k\ge 1$. Let $Q_1,\ldots,Q_q$ be hypersurfaces in $\P^n(\C)$ in $N-$subgeneral position with respect to $V$. Let $d_i=\deg Q_i\ (1\le i\le q)$ and let $d$ be the least common multiple of $d_1,\ldots,d_q$, i.e., $d=lcm(d_1,\ldots,d_q)$. Let $f:\C^m\to V$ be an algebraically nondegenerate meromorphic mapping. Then, for every $\epsilon >0$, 
$$\bigl \|\ \left (q-(N-k+1)(k+1)-\epsilon\right) T_f(r)\le\sum_{i=1}^q\frac{1}{d_i}N^{[M_0]}_{Q_i(f)}(r)+o(T_f(r)),$$
where $M_0=\left\lfloor\deg (V)^{k+1}e^kd^{k^2+k}p^k(2k+4)^kl^k\epsilon^{-k}\right\rfloor$ with $p=N-k+1$ and $l=(k+1)\cdot q!$.}

\noindent
Here, by the notation $\lfloor x\rfloor$ we denote the biggest integer not exceeding the real number $x$.

Subsequently, using the notion of distributive constant for families of hypersurfaces introduced by S. D. Quang \cite{Q22}, several authors improved Theorem B by lowering the total defect bound $(N-k+1)(k+1)$. We recall the following definitions.

\vskip0.2cm
\noindent
\textbf{Definition E}. {Let $V$ be a projective subvariety of $\P^n(\C)$ of dimension $k$ and $\mathcal Q=\{Q_1,\ldots,Q_q\}$ a family of hypersurfaces in $\P^n(\C)$.

(a) The family $\mathcal Q$ is said to satisfy the weak B\'ezout property if 
$$\operatorname{codim}_V\bigl(\bigcap_{j\in I\cup J}Q_j\cap V\bigl) \leq 2$$
for every subsets $I,J\subset\{1,\ldots,q\}$ with $\operatorname{codim}_V\bigl(\bigcap_{j\in I}Q_j\bigl)=\operatorname{codim}_V\bigl(\bigcap_{j\in J}Q_j\bigl)=1$.

(b) The family $\mathcal Q$ is said to satisfy the B\'ezout property if 
$$\operatorname{codim}_V\bigl(\bigcap_{j\in I\cup J}Q_j\cap V\bigl)\le \operatorname{codim}_V\bigl(\bigcap_{j\in I}Q_j\cap V\bigl) + \operatorname{codim}_V\bigl(\bigcap_{j\in J}Q_j\cap V\bigl)$$
 for every subsets $I,J \subset \{1,\ldots,q\}.$}

We denote by $\lceil x\rceil$ the smallest integer not smaller than the real number $x$. For convenient, we set
$$ u(d,k,v,\tau,\epsilon):=\lceil \tau(2k+1)(k+1)d^kv(\tau(k+1)+\epsilon)\epsilon^{-1}\rceil $$
and
$$L(d,k,v,\tau,\epsilon):=\lfloor d^{k^2+k}v^{k+1}\tau^ke^k(2k+5)^k(\tau(k+1)\epsilon^{-1}+1)^k\rfloor$$
for positive constants $d,k,v,\tau,\epsilon$. It is easy to see that the function $\frac{L-1}{u}$ is increasing in the variable $\tau>0$.

Recently, G. Dethloff and S. D. Quang \cite{Q26} proved the second main theorem for the case of families of hypersurfaces satisfying the weak B\'ezout property as follows.

\vskip0.2cm
\noindent
\textbf{Theorem F} (reformulation, \cite{Q26}) {\it  Let $V\subset\P^n(\C)$ be a smooth complex projective variety of dimension $k\ge 1$. Let $\mathcal Q=\{Q_1,\ldots,Q_q\}$ be a family of hypersurfaces of $\P^n(\C)$ in $N-$subgeneral position with respect to $V$. Let $d_i=\deg Q_i\ (1\le i\le q)$ and $d=lcm(d_1,\ldots,d_q)$. Let $f:\C^m\to V$ be an algebraically nondegenerate meromorphic mapping. Assume that the family $\mathcal Q$ satisfies the weak B\'ezout property. Then, for every $\epsilon >0$, 
$$\biggl \|\ \left (q-\left\lfloor\frac{N-k+3}{2}\right\rfloor-\tau_0(k+1)-\epsilon\right) T_f(r)\le\sum_{i=1}^q\frac{1}{d_i}N^{[L_0-1]}(r,\nu_{Q_i,f})+o(T_f(r)),$$
where $\tau_0=\frac{N-k+2}{2}$ and $L_0=L(d,k,\deg V,\tau,\epsilon)$.}

Then the above theorem, the total defect $\left\lfloor\tau_0+\frac{1}{2}\right\rfloor-\tau_0(k+1)$, which is smaller than $(N-k+1)(k+1)$ in generally. Also the truncation level $L_0-1$ does not depend on $q$ as in Theorem D.

Assuming that the family $\mathcal Q$ satisfies the B\'ezout property, L. Shi and Q. M. Yan \cite{LY} established a second main theorem with total defect 
$$\min\left\{N-k,\ \frac{N-k}{2}+1\right\} - p(k+1),$$
where the truncation level $M_0$ depends on $q$. Subsequently, G. Heier and A. Levin \cite{HL} obtained a second main theorem with total defect $\frac{3}{2}(2N-k+1)$, but without truncation. This bound is currently the best known, though it still remains far from the expected value $2N-k+1$.

The first purpose of this paper is to establish a second main theorem for hypersurfaces in subgeneral position, achieving a slightly smaller total defect than in previous results, with an explicit estimate for the truncation level. Our first result is as follows.

\begin{theorem}\label{1.1new}
Let $V\subset\P^n(\C)$ be a smooth complex projective variety of dimension $k\ge 1$. Let $\mathcal Q=\{Q_1,\ldots,Q_q\}$ be a family of hypersurfaces of $\P^n(\C)$ in $N-$subgeneral position with respect to $V\ (N>k)$. Let $d_i=\deg Q_i\ (1\le i\le q)$ and $d=lcm(d_1,\ldots,d_q)$. Let $f:\C\to V$ be an algebraically nondegenerate meromorphic mapping. Assume that the family $\mathcal Q$ satisfies the B\'ezout property. Then, for every $\epsilon >0$, 
$$\biggl \|\ \left (q -\left\lceil\frac{(N-k)\tau}{\tau-1}\right\rceil +1-\tau(k+1)-\epsilon\right) T_f(r)\le\sum_{i=1}^q\frac{1}{d_i}N^{[L_1-1]}(r,\nu_{(Q_i,f)})+o(T_f(r)),$$
where $L_1=L(d,k,\deg V,\tau,\epsilon)$ with 
$$\tau=\begin{cases}
1 + \sqrt{\frac{N-k}{k+1}}&\text{ if }\ k<N< \frac{5k+1}{4},\\
1+\frac{2(N-k)}{k+1}&\text{ if }\ \frac{5k+1}{4}\le N.
\end{cases}.$$
\end{theorem}

\textbf{\textit{Remark.}}
\begin{enumerate}
\item If $V=\P^n(\C)$, then any family of hypersurfaces $\mathcal Q=\{Q_1,\ldots,Q_q\}$ satisfies the B\'ezout property.

\item In the above result, the truncation level $L_1-1$ is independent of $q$. To obtain an estimate of the truncation level independent of $q$, we employ a new bound for the Chow weight from \cite{Q22}, together with a technique introduced in \cite{Q25} to control the error terms arising in the estimation of Hilbert weights.

\item The total defect obtained from Theorem \ref{1.1new} is bounded by
$$
\left\lceil\frac{(N-k)\tau}{\tau-1}\right\rceil -1+\tau(k+1).
$$
A direct computation shows that: if $k+1<N<\frac{5k+1}{4}$ then
\begin{align*}
\left\lceil\frac{(N-k)\tau}{\tau-1}\right\rceil -1+\tau(k+1)&\le 2\sqrt{(N-k)(k+1)}+N+1\\
&<\frac{3}{2}\left(2N-k+1\right);
\end{align*}
if $N\ge \frac{5k+1}{4}$ then
$$\left\lceil\frac{(N-k)\tau}{\tau-1}\right\rceil -1+\tau(k+1)\le 3N-2k+\left\lfloor\frac{k}{2}\right \rfloor+1.$$ 
\end{enumerate}

The second aim of this paper is to extend the above result to holomorphic curves from an annulus
$$
\A(R_0)=\left\{z\in\C:\ \frac{1}{R_0}<|z|<R_0\right\}, \quad (1<R_0\le +\infty),
$$
into a projective subvariety of $\P^n(\C)$. The result is as follows.
\begin{theorem}\label{1.2new}
Let $V\subset\P^n(\C)$ be a smooth complex projective variety of dimension $k\ge 1$. Let $\mathcal Q=\{Q_1,\ldots,Q_q\}$ be a family of hypersurfaces of $\P^n(\C)$ in $N-$subgeneral position with respect to $V\ (N>k)$. Let $d_i=\deg Q_i\ (1\le i\le q)$ and $d=lcm(d_1,\ldots,d_q)$.  Let $f$ be an algebraically nondegenerate holomorphic curve from $\A(R_0)\ (1<R_0\le +\infty)$ into $V$ and $\epsilon$ a positive constant.

(a)  If the family $\mathcal Q$ satisfies the weak B\'ezout property then
$$ \left (q-\left\lfloor\frac{N-k+3}{2}\right\rfloor-\tau_0(k+1)-\epsilon\right)T_0(r,f)\le\sum_{i=1}^q\frac{1}{d_i}N^{[L_0-1]}(r,\nu_{(Q_i,f)})+S_f(r),$$
where $\tau_0=\frac{N-k+2}{2}$ and $L_0=L(d,k,\deg V,\tau_0,\epsilon)$.

(b) If the family $\mathcal Q$ satisfies the B\'ezout property then
$$ \left (q -\left\lceil\frac{(N-k)\tau_1}{\tau_1-1}\right\rceil +1-\tau_1(k+1)-\epsilon\right)T_0(r,f)\le\sum_{i=1}^q\frac{1}{d_i}N^{[L_1-1]}(r,\nu_{(Q_i,f)})+S_f(r),$$
where $L_1=L(d,k,\deg V,\tau_1,\epsilon)$ with
$$\tau_1=\begin{cases}
1 + \sqrt{\frac{N-k}{k+1}}&\text{ if }\ k<N< \frac{5k+1}{4},\\
1+\frac{2(N-k)}{k+1}&\text{ if }\ \frac{5k+1}{4}\le N.
\end{cases}.$$
\end{theorem}

Let $f$ be a holomorphic curve from the complex disc $\Delta(R)=\{z\in\C:\ |z|<R\}$ into $\P^n(\C)$, and let $\Omega$ denote the Fubini--Study form on $\P^n(\C)$. The characteristic function of $f$ is defined by
$$
T_f(r):=\int_{0}^{r}\frac{dt}{t}\int_{|z|<t} f^*\Omega.
$$
Following M. Ru and N. Sibony \cite{RS}, the growth index of $f$ is defined as
$$
c_f:=\inf\left\{c>0:\ \int_{0}^{R}\exp\bigl(cT_f(r)\bigr)\,dr=+\infty\right\}.
$$
For convenience, we set $c_f:=+\infty$ if the above set is empty.

The third aim of this paper is to establish the following second main theorem.

\begin{theorem}\label{1.3new}
Let $V\subset\P^n(\C)$ be a smooth complex projective variety of dimension $k\ge 1$. Let $\mathcal Q=\{Q_1,\ldots,Q_q\}$ be a family of hypersurfaces of $\P^n(\C)$ in $N-$subgeneral position with respect to $V\ (N>k)$. Let $d_i=\deg Q_i\ (1\le i\le q)$ and $d=lcm(d_1,\ldots,d_q)$. Let $f$ be an algebraically nondegenerate holomorphic curve from $B(R)\ (1<R\le +\infty)$ into $V$ and $\epsilon$ a positive constant. Then

(a)  If the family $\mathcal Q$ satisfies the weak B\'ezout property then
\begin{align*}
\biggl \|\ &\left (q-\left\lfloor\frac{N-k+3}{2}\right\rfloor-\frac{(N-k+2)(k+1)}{2}-\epsilon\right)T_f(r)\\
&\le\sum_{i=1}^q\frac{1}{d_i}N^{[L_0-1]}(r,\nu_{(Q_i,f)})+\frac{(\tau_0(k+1)+\epsilon)(c_f+\epsilon')(L_0-1)}{2du}T_f (r), \ \forall\epsilon'>0,
\end{align*}
where $\tau_0=\frac{N-k+2}{2},$ $L_0=L(d,k,\deg V,\tau_0,\epsilon)$ and $u_0=u(d,k,\deg V,\tau_0,\epsilon).$

(b) If the family $\mathcal Q$ satisfies the B\'ezout property then
\begin{align*}
\biggl \|\ &\left (q -\left\lceil\frac{(N-k)\tau_1}{\tau_1-1}\right\rceil +1-\tau_1(k+1)-\epsilon\right)T_f(r)\\
&\le\sum_{i=1}^q\frac{1}{d_i}N^{[L_1-1]}(r,\nu_{(Q_i,f)})+\frac{(\tau_1(k+1)+\epsilon)(c_f+\epsilon')(L_1-1)}{2du_1}T_f (r),\  \forall\epsilon'>0,
\end{align*}
where $L_1=L(d,k,\deg V,\tau_1,\epsilon)$ and $u_1=u(d,k,\deg V,\tau_1,\epsilon)$ with
$$\tau_1=\begin{cases}
1 + \sqrt{\frac{N-k}{k+1}}&\text{ if }\ k<N< \frac{5k+1}{4},\\
1+\frac{2(N-k)}{k+1}&\text{ if }\ \frac{5k+1}{4}\le N.
\end{cases}.$$
\end{theorem}
In this setting, the notation $\|\ P$ means that $P$ holds for all $r\in(0,R)$ outside a set $S\subset(0,R)$ satisfying
$$
\int_S \exp\bigl((c_f+\epsilon)T_f(r)\bigr)\,dr < +\infty
$$
for some $\epsilon>0$.

Let $M$ be a complete K\"ahler manifold of dimension $m$, and let $f: M \to V$ be a meromorphic mapping. Denote by $\Omega_f$ the pull-back of the Fubini--Study form $\Omega$ on $\P^n(\C)$ by $f$. Let $Q$ be a hypersurface of degree $d$ in $\P^n(\C)$ with $f(M)\not\subset Q$. 

Following Fujimoto \cite{F85}, the non-integrated defect of $f$ with respect to $Q$, truncated to level $\mu_0$, is defined by
$$
\delta_f^{[\mu_0]}(Q):= 1 - \inf\{\eta \ge 0:\ \eta \text{ satisfies }(*)\}.
$$
Here, condition $(*)$ means that there exists a bounded non-negative continuous function $h$ on $M$, whose zeros have order at least $\min\{\nu_f(Q),\mu_0\}$, such that
$$
d\eta\,\Omega_f + \frac{\sqrt{-1}}{2\pi}\partial\bar\partial \log h^2 \ge [\min\{\nu_f(Q),\mu_0\}].
$$

We now extend the above second main theorems to non-integrated defect relations for holomorphic curves from K\"ahler manifolds into projective varieties, with respect to families of hypersurfaces in subgeneral position.

\begin{theorem}\label{1.4new}
Let $M$ be an $m$-dimensional complete K\"ahler manifold  and $\omega$ be a K\"ahler form of $M.$  Assume that the universal covering of $M$ is biholomorphic to a ball in $\mathbb C^m.$ Let $f$ be an algebraically nondegenerate meromorphic map of $M$ into a subvariety $V$ of dimension $k$ in $\P^n(\C)\ (N>k)$ and $\epsilon>0$. Let $\mathcal Q=\{Q_1,\ldots,Q_q\}$ be a family of $q$ hypersurfaces in $\P^n(\C)$ in $N$-subgeneral position with respect to $V$ and $d = lcm\{\deg Q_1,\ldots,\deg Q_q\}$. Assume that for some $\rho \ge 0,$ there exists a bounded continuous function $h \geq 0$ on $M$ such that
$$\rho\Omega_f + \mathrm{dd^c}\log h^2 \geq \mathrm{ric}\ \omega.$$

(a)  If the family $\mathcal Q$ satisfies the weak B\'ezout property then
$$\sum_{j=1}^q {\delta}^{[L_0-1]}_{f} (Q_j) \le \left\lfloor\frac{N-k+3}{2}\right\rfloor+\frac{(N-k+2)(k+1)}{2}+\epsilon+\dfrac{\rho(L_0-1)(\Delta_{\mathcal Q,V}(k+1)+\epsilon)}{u_0d},$$
where $L_0=L(d,k,\deg V,\tau_0,\epsilon)$ and $u_0=u(d,k,\deg V,\tau_0,\epsilon)$ with $\tau_0=\frac{N-k+2}{2}.$ 

(b) If the family $\mathcal Q$ satisfies the B\'ezout property then
$$\sum_{j=1}^q {\delta}^{[L_1-1]}_{f} (Q_j) \le \left\lceil\frac{(N-k)\tau_1}{\tau_1-1}\right\rceil -1+\tau_1(k+1)+\epsilon +\dfrac{\rho(L-1)(\Delta_{\mathcal Q,V}(k+1)+\epsilon)}{ud},$$
where $L_1=L(d,k,\deg V,\tau_1,\epsilon)$ and $u_1=u(d,k,\deg V,\tau_1,\epsilon)$ with 
$$\tau_1=\begin{cases}
1 + \sqrt{\frac{N-k}{k+1}}&\text{ if }\ k<N< \frac{5k+1}{4},\\
1+\frac{2(N-k)}{k+1}&\text{ if }\ \frac{5k+1}{4}\le N.
\end{cases}.$$
\end{theorem}

This result also generalizes those in \cite{Q17,Q21,RSo,Y13}.

In the final section, we establish number-theoretic counterparts of Theorem D and Theorem \ref{1.1new} via Vojta’s dictionary, relating Nevanlinna theory to Diophantine approximation (see \cite{V87,V11}). In particular, we study the degenerated Schmidt’s subspace theorem for families of hypersurfaces in subgeneral position. Using the notation of Section 5, our main result is as follows.

\begin{theorem}\label{1.5new} 
Let $k$ be a number field, $S$ be a finite set of places of $k$ and let $V$ be an irreducible projective subvariety of $\P^N$ of dimension $n$ defined over of $k$. Let $\mathcal Q=\{Q_1,\ldots, Q_q\}$ be a family of $q$ homogeneous polynomials of $\bar k[x_0,\ldots,x_N]$ in $\ell$-subgeneral position with respect to $V(\bar k)\ (\ell>n)$. For each $\epsilon >0$, we have: 
\begin{itemize}
\item[(a)]  If the family $\mathcal Q$ satisfies the weak B\'ezout property then
$$\sum_{v\in S}\sum_{j=1}^q\dfrac{\lambda_{Q_j,v}({\bf x})}{\deg Q_j}\le \left(\left\lfloor\frac{\ell-n+3}{2}\right\rfloor+\frac{(\ell-n+2)(n+1)}{2}+\epsilon\right)h({\bf x})$$
for all ${\bf x}\in\P^N(k)$ outside a union of closed proper subvarieties of $V$.
\item [(b)]  If the family $\mathcal Q$ satisfies the B\'ezout property then
$$\sum_{v\in S}\sum_{j=1}^q\dfrac{\lambda_{Q_j,v}({\bf x})}{\deg Q_j}\le \left(\left\lceil\frac{(\ell-n)\tau_1}{\tau_1-1}\right\rceil -1+\tau_1(n+1)+\epsilon \right)h({\bf x})$$
for all ${\bf x}\in\P^N(k)$ outside a union of closed proper subvarieties of $V$, where
$$\tau_1=\begin{cases}
1 + \sqrt{\frac{\ell-n}{n+1}}&\text{ if }\ n<\ell\le \frac{5n+1}{4},\\
1+\frac{2(\ell-n)}{n+1}&\text{ if }\ \frac{5n+1}{4}<\ell.
\end{cases}.$$
\end{itemize}
\end{theorem}

\section{Notations and Preliminaries}

\noindent
{\bf 2.1. Nevanlinna functions on $\C$ and $\Delta (R)$}

We set $\Delta(R)=\{z\in\C\ |\ |z|<R\}\ (0<R\le +\infty)$. If $R=+\infty$ then we regard $\Delta(R)$ as the complex plane $\C$.

For a divisor $\nu$ on $\Delta(R)$, we define the counting function of $\nu$ by
$$N(r,r_0,\nu)=\int\limits_{r_0}^r \dfrac {n(t)}{t}dt \quad (0<r_0<r<R),$$
where $n(t) =\sum\limits_{|z|\leq t} \nu (z).$ For a positive integer $M$, the counting function of $\nu$ truncated to level $M$ is defined by
$$ N^{[M]}(r,r_0,\nu):=N(r,r_0,\nu^{[M]}),$$
where $\nu^{[M]}(z)=\min\ \{M,\nu(z)\}.$

Throughout this paper, the positive number $r_0$ is always chosen fixed such that $r_0<R$ and $r_0=1$ if $R=+\infty.$ We just write $N(r,\nu)$ and $N^{[M]}(r,\nu)$ for $N(r,r_0,\nu)$ and $N^{[M]}(r,r_0,\nu)$ respectively. For a meromorphic function $\varphi$ on $\Delta(R)$, denote by $\nu_\varphi$ its divisor of zeros and set
$$N_{\varphi}(r)=N(r,\nu_{\varphi}), \ N_{\varphi}^{[M]}(r)=N^{[M]}(r,\nu_{\varphi})\ (r_0<r<R).$$

Now fix a homogeneous coordinates system $(x_0 : \dots : x_n)$ on $\mathbb P^n(\mathbb C)$ and consider a holomorphic curve $f : \Delta(R)\longrightarrow \mathbb P^n(\mathbb C)$ with a reduced representation $\tilde f = (f_0 , \ldots , f_n)$. Set $\|\tilde f \| = \big(|f_0|^2 + \cdots + |f_nN|^2\big)^{1/2}$.
The characteristic function of $f$ is defined by 
$$ T_f(r):=\int_{r_0}^r\dfrac{dt}{t^{2m-1}}\int\limits_{\Delta(t)}f^*\Omega, \ (0<r_0<r<R),$$
where $\Omega$ is the Fubini-Study form on $\P^n(\C)$. By Jensen's formula, we have
\begin{align*}
T_f(r,r_0)=\frac{1}{2\pi}\int\limits_{0}^{2\pi}\log\Vert \tilde f(re^{i\theta}) \Vert d\theta -
\frac{1}{2\pi}\int\limits_{0}^{2\pi}\log\Vert \tilde f(r_0e^{i\theta}) \Vert d\theta+O(1), \text{ (as $r\rightarrow R$)}.
\end{align*}

Let $Q$ be a hypersurface in $\P^n(\C)$ of degree $d$, which has a defining homogeneous polynomial $D$ of the form
$$D({\bf x})=\sum_{I\in\mathcal T_d}a_I{\bf x}^I, $$
where $\mathcal T_d=\{(i_0,\ldots,i_N)\in\mathbb Z_{\geqslant 0}^{N+1}\ ;\ i_0+\cdots +i_N=d\}$, ${\bf x} =(x_0,\ldots,x_N)$, ${\bf x}^I=x_0^{i_0}\cdots x_N^{i_N}$ with $I=(i_0,\ldots,i_N)\in\mathcal T_d$ and $a_I\ (I\in\mathcal T_d)$ are constants, not all zeros, i.e., 
$$Q=\{(x_0:\cdots:x_n)\in\P^n(\C)\ |\ D(x_0,\ldots,x_n)=0\}.$$
The proximity function of $f$ with respect to $Q$, denoted by $m_f (r,Q)$, is defined by
$$m_f (r,Q)=\frac{1}{2\pi}\int\limits_{0}^{2\pi}\log\dfrac{\|\tilde f(re^{i\theta})\|^d}{|Q(\tilde f(re^{i\theta}))|}d\theta-\frac{1}{2\pi}\int\limits_{0}^{2\pi}\log\dfrac{\|\tilde f(r_0e^{i\theta})\|^d}{|Q(\tilde f(r_0e^{i\theta}))|}d\theta,$$
where $Q(\tilde f)=Q(f_0,\ldots,f_n)$. This definition is independent of the choice of the reduced representation of $f$. 

We denote by $\nu_{(Q,f)}$ the pullback of the divisor $Q$ by $f$. Then, $\nu_{(Q,f)}$ identifies with the divisor of zeros $\nu^0_{D(\tilde f)}$ of the function $D(\tilde f)$. By Jensen's formula, we have
$$N(r,\nu_{(Q,f)})=N_{D(\tilde f)}(r)=\frac{1}{2\pi}\int\limits_{0}^{2\pi}\log |D(\tilde f(re^{i\theta}))|d\theta-\frac{1}{2\pi}\int\limits_{0}^{2\pi}\log |D(\tilde f(r_0e^{i\theta}))|d\theta.$$
The first main theorem in Nevanlinna theory for meromorphic mappings and hypersurfaces is stated as follows.
$$ dT_f(r,r_0)=m_f (r,Q)+N(r,\nu_{(Q,f)})+O(1).$$

\noindent
{\bf 2.2. Nevanlinna functions on annuli}

For a divisor $\nu$ on $\A (R_0)$ and  for a positive integer $M$ (maybe $M= + \infty$), we define the counting function of $\nu$ as follows
$$N_0^{[M]}(r,\nu)=\int\limits_{\frac{1}{r}}^1 \dfrac {n_0^{[M]}(t)}{t}dt +\int\limits_1^r \dfrac {n_0^{[M]}(t)}{t}dt \quad (1<r<R_0),$$
\begin{align*}
\text{ where }n_0^{[M]}(t)=\begin{cases}
\sum\limits_{1\le |z|\le t}\min\{M,\nu (z)\}&\text{ if }1\le t<R_0,\\
\sum\limits_{t\le |z|<1}\min\{M,\nu (z)\}&\text{ if }\dfrac{1}{R_0}<t< 1.
\end{cases}
\end{align*}
For brevity we will omit the character $^{[M]}$ if $M= +\infty$.

For a meromorphic function  $\varphi $ on $\A(R_0)$, we define the proximity function of $\varphi$ by
$$ m_0(r,\varphi)=\dfrac{1}{2\pi}\int\limits_{0}^{2\pi}\log^+|\varphi(\dfrac{1}{r}e^{i\theta})| d\theta +\dfrac{1}{2\pi}\int\limits_{0}^{2\pi}\log^+|\varphi(re^{i\theta})| d\theta- \dfrac{1}{\pi}\int\limits_{0}^{2\pi}\log^+|\varphi(e^{i\theta})| d\theta.$$
The Nevanlinna characteristic function of $f$ is defined by
$$ T_0(r,\varphi)=m_0(r,\varphi)+N_0(r,\nu^\infty_\varphi). $$
Throughout this paper, we denote by $S_f(r)$ quantities satisfying
\begin{enumerate}
\item[(i)] in the case $R_0=+\infty$,
$$ S_\varphi(r)=O(\log (rT_0(r,\varphi)))$$ 
for $r\in (1,+\infty)$ except for a set $E$ such that $\int_{E}r^{\lambda -1}dr <+\infty$ for some $\lambda \ge 0$,
\item [(ii)] in the case $R_0<+\infty$,
$$ S_\varphi(r)=O(\log (\dfrac{T_0(r,\varphi)}{R_0-r}))\text{ as }r\longrightarrow R_0$$ 
for $r\in (1,R_0)$ except for a set $E'$ such that $\int_{E'}\dfrac{dr}{(R_0-r)^{\lambda +1}} <+\infty$ for some $\lambda \ge 0$.
\end{enumerate}
\begin{lemma}[Lemma on logarithmic derivatives \cite{CD12,KK05a,KK05b}]\label{2.1}
Let $\varphi$ be a nonzero meromorphic function on $\A (R_0)$. Then for each positive integer $k$ we have
$$ m_0\left(r,\dfrac{\varphi^{(k)}}{\varphi}\right)=S_\varphi(r)\ (1\le r <R_0). $$
\end{lemma}
Here by $\varphi^{(k)}$ we denote the derivative of order $k$ of the function $f$.

\begin{theorem}[First main theorem \cite{CD12,KK05a,KK05b}]
Let $f$ be a meromorphic function on $\A (R_0)$. Then for each $a\in\C$ we have
$$ T_0(r,\varphi)=T_0\left(r,\dfrac{1}{\varphi-a}\right)+S_\varphi(r)\ (1\le r <R_0).$$
\end{theorem}

Now we consider a holomorphic curve $f$ from an annulus $\A (R_0)$ into $\P^n(\C)$ with a reduced representation $\tilde f = (f_0,\ldots,f_n)$. The characteristic function of $f$ is defined by
$$ T_0(r,f)=\dfrac{1}{2\pi}\int\limits_{0}^{2\pi}\log \|\tilde f(re^{i\theta})\| d\theta +\dfrac{1}{2\pi}\int\limits_{0}^{2\pi}\log \|\tilde f(\frac{1}{r}e^{i\theta})\| d\theta-\dfrac{1}{\pi}\int\limits_{0}^{2\pi}\log \|\tilde f(e^{i\theta})\| d\theta.$$
Let $Q$ be a hypersurface in $\P^n(\C)$ of degree $d$ with a defining homogeneous polynomial $D$ as in the subsection 2.1
The proximity function of $f$ with respect to $Q$ is defined by
\begin{align*}
 m_0(r,f,Q)=& \dfrac{1}{2\pi}\int\limits_{0}^{2\pi}\log \dfrac{\|\tilde f(re^{i\theta})\|^d}{|D(\tilde f)(re^{i\theta})|}d\theta +\dfrac{1}{2\pi}\int\limits_{0}^{2\pi}\log \dfrac{\|\tilde f(\frac{1}{r}e^{i\theta})\|^d}{|D(\tilde f)(\frac{1}{r}e^{i\theta})|} d\theta-\dfrac{1}{\pi}\int\limits_{0}^{2\pi}\log \dfrac{\|\tilde f(e^{i\theta})\|}{|D(\tilde f)(e^{i\theta})|} d\theta.
\end{align*}
By Jensen's formula, we have the First Main Theorem as follows
$$dT_0(r,f)=m_0(r,f,Q)+N_0(r,\nu_{(Q,f)}).$$

\section{Second main theorem for holomorphic curves and arbitrary families of hypersurfaces with explicit truncation levels}

To prove our results, we first establish a second main theorem for holomorphic curves and arbitrary families of hypersurfaces with an explicit truncation level. Since the arguments for holomorphic curves on $\C$ and on annuli are analogous (the latter being slightly more involved), we only prove the second main theorem for  the case of holomorphic curves from $\A(R_0)$ and treat arbitrary families of hypersurfaces in terms of their distributive constant. We recall this notion below.

\begin{definition}[{see \cite{Q22}}]
Let $X$ be a projective subvariety of $\P^n(\C)$ of dimension $k\le n$ and $\mathcal Q=\{Q_1,\ldots,Q_q\}$ a family of hypersurfaces in $\P^n(\C)$. The distributive constant of the family $\mathcal Q$ with respect to $X$ is defined by
$$ \Delta_{\mathcal Q,X}:=\underset{\varnothing\ne\Gamma\subset\{1,\ldots,q\}}\max\frac{\sharp\Gamma}{k-\dim\left (\bigcap_{j\in\Gamma} Q_j\right )\cap X}.$$
\end{definition}
\noindent
Here and throughout this paper, for convenience we adopt the convention that
$$\dim \varnothing=-1,$$
and define
$$\mathrm{codim}_X S=\dim X-\dim S$$
for every analytic subset $S\subset X$.
 
The following lemma due to S. D. Quang and D. P. An \cite{QA} is stated for the case of meromorphic mappings on $\C^m$ but it automatically holds for the case of holomorphic curves on annuli.  
\begin{lemma}[{cf. \cite[Lemma 4]{QA}}]\label{3.2}
Let $\{Q_i\}_{i\in R}$ be a set of hypersurfaces in $\P^n(\C)$ of the common degree $d$ and let $f$ be a holomorphic curve of $\mathbb A(R_0)$ into $\P^n(\C)$. Assume that $\bigcap_{i\in R}Q_i\cap V=\varnothing$. Then there exist positive constants $\alpha$ and $\beta$ such that
$$\alpha \|\tilde f\|^d \le  \max_{i\in R}|D_i(f)|\le \beta \|\tilde f\|^d,$$
where $\tilde f$ is a reduced representation of $f$ and $D_i\ (i\in R)$ is a homogeneous polynomial defining $Q_i$ of degree $d$.
\end{lemma} 

We first prove the following general form of second main theorem.
\begin{theorem}\label{3.3}
Let $f$ be a linearly nondegenerate holomorphic curve from $\A(R_0)\ (1<R_0\le +\infty)$ into $\P^n(\C)$ with a reduced representation $\tilde f=(f_0,\ldots,f_n)$. Let $\{L_i\}_{i=1}^q$ be a set of linear form in $n+1$ variables. Then, we have
\begin{align*}
\frac{1}{2\pi}\int_{0}^{2\pi}\max_{K_1}\sum_{i\in K_1}\log\frac{\|\tilde f(re^{i\theta})\|}{|L_i(\tilde f(re^{i\theta}))|}d\theta&+\frac{1}{2\pi}\int_{0}^{2\pi}\max_{K_2}\sum_{i\in K_2}\log\frac{\|\tilde f(\frac{1}{r}e^{i\theta})\|}{|L_i(\tilde f(\frac{1}{r}e^{i\theta}))|}d\theta\\
&\le (n+1)T_0(r,f)-N_0(r,\nu_W)+S_f(r),
\end{align*}
where the maximums are taken over all subsets $K_t\ (t=1,2)$ of $\{1,\ldots,q\}$ so that $\{L_i|i\in K_t\}$ is linear independent, and $W$ is the wronskian of $\tilde f$, i.e., $W=\det (f_j^{(i)};0\le i,j\le n)$.
\end{theorem}
\begin{proof}
By adding more linear forms if necessary, we suppose that rank$\{L_i;1\le i\le q\}=n+1$. Denote by $\mathcal L$ the set of all subsets $K\subset\{1,\ldots,q\}$ such that $\sharp K=n+1$ and $\{L_i|i\in K\}$ is linear independent. Again, by adding more linear forms, we may assume that the maximums in the desired inequality of the theorem are taken over all subsets $K_t\in\mathcal L\ (t=1,2)$. Note that, for each $K\in\mathcal L$, there is a positive constant $C_{K}$ such that
$$|W|=C_{K}|\det (H_j(\tilde f)^{(i)};j\in K,0\le i\le n)|.$$
We set 
$$S_K=\log^+\left(\frac{|\det (H_j(\tilde f)^{(i)};j\in K,0\le i\le n)|}{\prod_{j\in K}|H_j(\tilde f)|}\right).$$
By Lemma \ref{2.1}, we have
\begin{align}\label{1}
\frac{1}{2\pi}\int_{0}^{2\pi}S_K(re^{i\theta})d\theta+\frac{1}{2\pi}\int_{0}^{2\pi}S_K\left(\frac{1}{r}e^{i\theta}\right)d\theta=S_f(r).
\end{align}
On the other hand, one has
$$ \max_{K_t\in\mathcal L}\sum_{i\in K_t}\log\frac{\|\tilde f(z)\|\cdot |W(z)|}{|L_i(\tilde f(z))|}\le \sum_{K\in\mathcal L}S_K+(n+1)\log\|\tilde f(z)\|+O(1)\ (t=1,2).$$
Integrating both sides of this inequalities and using Jensen's formula and (\ref{1}), we get 
\begin{align*}
\frac{1}{2\pi}\int_{0}^{2\pi}\max_{K_1}\sum_{i\in K_1}\log\frac{\|\tilde f(re^{i\theta})\|}{|L_i(\tilde f(re^{i\theta}))|}d\theta&+\frac{1}{2\pi}\int_{0}^{2\pi}\max_{K_2}\sum_{i\in K_2}\log\frac{\|\tilde f(\frac{1}{r}e^{i\theta})\|}{|L_i(\tilde f(\frac{1}{r}e^{i\theta}))|}d\theta\\
&+N_0(r,\nu_W)\le (n+1)T_0(r,f)+S_f(r).
\end{align*}
The theorem is proved.
\end{proof}

Let $X \subset \P^n(\C)$ be a projective subvariety of dimension $k$ and degree $\delta$. For $\mathbf a=(a_0,\ldots,a_n)\in \mathbb Z^{n+1}_{\ge 0}$, write $\mathbf x^{\mathbf a}=x_0^{a_0}\cdots x_n^{a_n}$. Let $I_X \subset \C[x_0,\ldots,x_n]$ be the prime ideal defining $X$, and denote by $\C[x_0,\ldots,x_n]_u$ the vector space of homogeneous polynomials of degree $u$ (including the zero polynomial). For $u\ge 1$, set $(I_X)_u := I_X \cap \C[x_0,\ldots,x_n]_u$ and define the Hilbert function by
\begin{align*}
H_X(u):=\dim \C[x_0,\ldots,x_n]_u/(I_X)_u.
\end{align*}

Let $\mathbf c=(c_0,\ldots,c_n)\in \mathbb R^{n+1}_{\ge 0}$, and let $e_X(\mathbf c)$ denote the Chow weight of $X$ with respect to $\mathbf c$. The $u$-th Hilbert weight $S_X(u,\mathbf c)$ is defined by
\begin{align*}
S_X(u,\mathbf c):=\max \sum_{i=1}^{H_X(u)} \mathbf a_i \cdot \mathbf c,
\end{align*}
where the maximum is taken over all collections of monomials $\mathbf x^{\mathbf a_1},\ldots,\mathbf x^{\mathbf a_{H_X(u)}}$ whose residue classes modulo $I_X$ form a basis of $\C[x_0,\ldots,x_n]_u/(I_X)_u$.

The following theorem is due to J. Evertse and R. Ferretti \cite{EF01}.
\begin{theorem}[{see \cite[Theorem 4.1]{EF01}}]\label{4.1}
Let $X\subset\P^n(\C)$ be an algebraic variety of dimension $k$ and degree $\delta$. Let $u>\delta$ be an integer and let ${\bf c}=(c_0,\ldots,c_n)\in\mathbb R^{n+1}_{\geqslant 0}$.
Then
$$ \dfrac{1}{uH_X(u)}S_X(u,{\bf c})\ge\dfrac{1}{(k+1)\delta}e_X({\bf c})-\dfrac{(2k+1)\delta}{u}\cdot\left (\max_{i=0,\ldots,n}c_i\right).$$
\end{theorem}

The following lemma is due to the first author in \cite{Q22b}.
\begin{lemma}[{see \cite[Lemma 3.2]{Q22b}}]\label{4.2}
Let $Y$ be a projective subvariety of $\P^R(\C)$ of dimension $k\ge 1$ and degree $\delta_Y$. Let $\ell\ (\ell\ge k+1)$ be an integer and let ${\bf c}=(c_0,\ldots,c_R)$ be a tuple of non-negative reals. Let $\mathcal H=\{H_0,\ldots,H_R\}$ be a set of hyperplanes in $\P^R(\C)$ defined by $H_{i}=\{y_{i}=0\}\ (0\le i\le R)$. Let $\{i_1,\ldots, i_\ell\}$ be a subset of $\{0,\ldots,R\}$ such that:
\begin{itemize}
\item[(1)] $c_{i_\ell}=\min\{c_{i_1},\ldots,c_{i_\ell}\}$,
\item[(2)] $Y\cap\bigcap_{j=1}^{\ell-1}H_{i_j}\ne \varnothing$, 
\item[(3)] and $Y\not\subset H_{i_j}$ for all $j=1,\ldots,\ell$.
\end{itemize}
Let $\Delta_{\mathcal H,Y}$ be the distributive constant of the family $\mathcal H=\{H_{i_j}\}_{j=1}^\ell$ with respect to $Y$. Then
$$e_Y({\bf c})\ge \frac{\delta_Y}{\Delta_{\mathcal H,Y}}(c_{i_1}+\cdots+c_{i_\ell}).$$
\end{lemma}

We now prove the following second main theorem for holomorphic curves from annuli into projective varieties and arbitrary families of hypersurfaces, which is an improvement of the recent result of L. Yang and Y.Zhu \cite{Zh}.

\begin{theorem}\label{1.1}
Let $V\subset\P^n(\C)$ be a smooth complex projective variety of dimension $k\ge 1$. Let $\mathcal Q=\{Q_1,\ldots,Q_q\}$ be a family of hypersurfaces of $\P^n(\C)$ with the distributive constant $\Delta_{\mathcal Q,V}$ with respect to $V$. Let $d_i=\deg Q_i\ (1\le i\le q)$, and $d=lcm(d_1,\ldots,d_q)$. Let $f$ be an algebraically nondegenerate holomorphic curve from $\A(R_0)\ (1<R_0\le +\infty)$ into $V$. Then, for every $\epsilon >0$, 
$$(q-\Delta_{\mathcal Q,V}(k+1)-\epsilon) T_0(r,f)\le\sum_{i=1}^q\frac{1}{d_i}N^{[L_0-1]}_0(r,\nu_{(Q_i,f)})+S_f(r),$$
where $L_0=L(d,k,\deg V,\Delta_{\mathcal Q,V},\epsilon)$.
\end{theorem}

\begin{proof} Let $D_1,\ldots,D_q$ be a homogeneous polynomial defining the hypersurfaces $Q_1,\ldots,Q_q$, respectively, where $\deg D_i=d_i\ (1\le i\le q)$. Replacing $D_i$ by $D^{d/d_i}$ if necessary, we may assume that $D_1,\ldots,D_q$ have the same degree $d$. It is suffice for us to consider the case where $q>\Delta_{\mathcal Q,V}(k+1)$. 

Denote by $\sigma_1,\ldots,\sigma_{n_0}$ all bijections from $\{0,\ldots,q-1\}$ into $\{1,\ldots,q\}$, where $n_0=q!$. For each $\sigma_i$, it is easy to see that $\bigcap_{j=0}^{q-2}Q_{\sigma_i(j)}\cap V=\varnothing$. Then there exists a smallest index $\ell_i\le q-2$ such that $\bigcap_{j=0}^{\ell_i}Q_{\sigma_i(j)}\cap V=\varnothing$. Hence $\bigcap_{j=0}^{\ell_i}Q_{\sigma_i(j)}\cap V=\varnothing$ for all $i=1,\ldots,n_0$.

By Lemma \ref{3.2}, there is a positive constant $A$, independent of $\sigma_i$, such that
$$ \|\tilde f (z)\|^d\le A\max_{0\le j\le \ell_i}|D_{\sigma_i(j)}(\tilde f)(z)|\ \forall i=1,\ldots,n_0.$$
Denote by $S(i)$ the set of all $z$ such that $D_j(\tilde f)(z)\ne 0$ for all $j=1,\ldots,q$ and
$$|D_{\sigma_i(0)}(\tilde f)(z)|\le |D_{\sigma_i(1)}(\tilde f)(z)|\le\cdots\le |D_{\sigma_i(q-1)}(\tilde f)(z)|.$$
Therefore, for every $z\in S(i)$, we have
$$\prod_{j=1}^q\dfrac{\|\tilde f (z)\|^d}{|D_j(\tilde f)(z)|}\le  C\prod_{j=0}^{\ell_j}\dfrac{\|\tilde f (z)\|^d}{|D_{\sigma_i(j)}(\tilde f)(z)|},$$
where $C$ is a positive constant (independent of $\sigma_i$).

Consider the mapping $\Phi$ from $V$ into $\P^{q-1}(\C)$ which maps the point ${\bf x}=(x_0:\cdots:x_n)\in V$ into the point
$$\Phi({\bf x})=(D_1(x):\cdots : D_{q}(x))\in\P^{q-1}(\C),$$
where $x=(x_0,\ldots,x_n)$. We set
$$\tilde\Phi(x)=(D_1(x),\ldots ,D_{q}(x)).$$
Let $Y=\Phi(V)$. Since $V\cap\bigcap_{j=1}^{q}Q_j=\varnothing$, $\Phi$ is a finite morphism on $V$ and $Y$ is a projective subvariety of $\P^{q-1}(\C)$ with $\dim Y=k$ and of degree
$$\delta:=\deg Y\le d^{k}.\deg V.$$ 
For every ${\bf a} = (a_1,\ldots,a_q)\in\mathbb Z^q_{\ge 0}$ and ${\bf y} = (y_1,\ldots,y_q)$ we denote ${\bf y}^{\bf a} = y_{1}^{a_{1}}\ldots y_{q}^{a_{q}}$. Let $u$ be a positive integer. We set $\xi_u:=\binom{q+u}{u}$ and define the $\C$-vector space
$$ Y_{u}:=\C[y_1,\ldots,y_q]_u/(I_{Y})_u \text{ and }n_u:=\dim Y_u-1.$$
We fix a basis $\{v_0,\ldots, v_{n_u}\}$ of $Y_u$ and consider the holomorphic map $F:\A(R_0)\rightarrow\P^{n_u}(\C)$ which has a reduced representation
$$ \tilde F=(v_0(\tilde\Phi\circ \tilde f),\ldots ,v_{n_u}(\tilde\Phi\circ \tilde f)):\A(R_0)\rightarrow \C^{n_u+1}. $$
Hence $F$ is linearly nondegenerate, since $f$ is algebraically nondegenerate.

Now, we fix a point $z\not\in\bigcup_{j=1}^q(D_j(\tilde f))^{-1}(0)$. Suppose that $z\in S(i_0)$. We define
$${\bf c}_z = (c_{1,z},\ldots,c_{q,z})\in\mathbb R^{q},$$
where
\begin{align}\label{3}
c_{j,z}:=\log\frac{\|\tilde f(z)\|^d}{|D_j(\tilde f(z))|}\text{ for } j=1, \ldots ,q.
\end{align}
We see that $c_{j,z}\ge 0$ for all $j$. By the definition of the Hilbert weight, there are ${\bf a}_{0,z}, \ldots ,{\bf a}_{n_u,z}\in\mathbb Z^{q}_{\ge 0}$ with
$$ {\bf a}_{i,z}=(a_{i1,z},\ldots, a_{iq,z})\text{ with }a_{is,z}\in\{0, \ldots ,u\}, $$
 such that the residue classes modulo $I(Y)_u$ of ${\bf y}^{{\bf a}_{0,z}}, \ldots ,{\bf y}^{{\bf a}_{n_u,z}}$ form a basis of $I_u(Y)$ and
\begin{align}\label{4}
S_Y(u,{\bf c}_z)=\sum_{i=0}^{n_u}{\bf a}_{i,z}\cdot{\bf c}_z.
\end{align}
We see that ${\bf y}^{{\bf a}_{i,z}}\in Y_u$ (modulo $I(Y)_u$). Then we may rewrite
$$ {\bf y}^{{\bf a}_{i,z}}=L_{i,z}(v_0,\ldots ,v_{n_u}), $$
where $L_{i,z}\ (0\le i\le n_u)$ are linearly independent linear forms with coefficients in $\C$. We have
\begin{align*}
\log\prod_{i=0}^{n_u} |L_{i,z}(\tilde F(z))|&=\log\prod_{i=0}^{n_u}\prod_{1\le j\le q}|D_j(\tilde f(z))|^{a_{ij,z}}\\
&=-S_Y(m,{\bf c}_z)+du(n_u+1)\log \|\tilde f(z)\| +O(u(n_u+1)).
\end{align*}
This implies that
\begin{align*}
\log\prod_{i=0}^{n_u}\dfrac{\|\tilde F(z)\|\cdot \|L_{i,z}\|}{|L_{i,z}(\tilde F(z))|}=&S_Y(u,{\bf c}_z)-du(n_u+1)\log \|\tilde f(z)\| \\
&+(n_u+1)\log \|\tilde F(z)\|+O(u(n_u+1)).
\end{align*}
Here we note that $L_{i,z}$ depends on $i$, $z$ and $u$, but the number of these linear forms is finite. We denote by $\mathcal L$ the set of all $L_{i,z}$ occurring in the above inequalities. Then,
\begin{align}\label{5}
\begin{split}
S_Y(u,{\bf c}_z)\le&\max_{\mathcal J\subset\mathcal L}\log\prod_{L\in \mathcal J}\dfrac{\|\tilde F(z)\|\cdot \|L\|}{|L(\tilde F(z))|}+du(n_u+1)\log \|\tilde f(z)\|\\
& -(n_u+1)\log \|\tilde F(z)\|+O(u(n_u+1)),
\end{split}
\end{align}
where the maximum is taken over all subsets $\mathcal J\subset\mathcal L$ with $\sharp\mathcal J=n_u+1$ and $\{L;L\in\mathcal J\}$ is linearly independent.
From Theorem \ref{4.1}, we have
\begin{align}\label{6}
\dfrac{1}{u(n_u+1)}S_Y(u,{\bf c}_z)\ge&\frac{1}{(k+1)\delta}e_Y({\bf c}_z)-\frac{(2k+1)\delta}{u}\max_{1\le j\le q}c_{j,z}.
\end{align}
One remarks that
\begin{align*}
\max_{1\le j\le q}c_{j,z}\le \sum_{1\le j\le q}\log\frac{\|\tilde f(z)\|^d}{|D_j(\tilde f(z))|}.
\end{align*}
Combining (\ref{5}), (\ref{6}) and the above remark, we get
\begin{align}\label{7}
\begin{split}
\frac{1}{(k+1)\delta}e_Y({\bf c}_z)\le &\dfrac{1}{u(n_u+1)}\left (\max_{\mathcal J\subset\mathcal L}\log\prod_{L\in \mathcal J}\dfrac{\|\tilde F(z)\|\cdot \|L\|}{|L(\tilde F(z))|}-(n_u+1)\log \|\tilde F(z)\|\right )\\
&+d\log \|\tilde f(z)\|+\frac{(2k+1)\delta}{u}\max_{1\le j\le q}c_{j,z}+O(1)\\
\le &\dfrac{1}{u(n_u+1)}\left (\max_{\mathcal J\subset\mathcal L}\log\prod_{L\in\mathcal J}\dfrac{\|\tilde F(z)\|\cdot \|L\|}{|L(\tilde F(z))|}-(n_u+1)\log \|\tilde F(z)\|\right )\\
&+d\log \|\tilde f(z)\|+\frac{(2k+1)\delta}{u}\sum_{1\le j\le q}\log\frac{\|\tilde f(z)\|^d}{|D_j(\tilde f(z))|}+O(1).
\end{split}
\end{align}
We have $V\cap\bigcap_{j=0}^{\ell_{i_0}}Q_{\sigma_{i_0}(j)}=\varnothing$. Then by Theorem \ref{4.2} and (\ref{7}), we have
\begin{align}\label{8}
\begin{split}
e_Y({\bf c}_z)&\ge \frac{\delta}{\Delta_{\mathcal Q,V}}\cdot(c_{\sigma_{i_0}(0),z}+\cdots +c_{\sigma_{i_0}(l_{i_0}),z})=\frac{\delta}{\Delta_{\mathcal Q,V}}\cdot\log\prod_{j=0}^{l_{i_0}}\frac{\|\tilde f(z)\|^d}{|Q_{\sigma_{i_0}(j)}(\tilde f(z))|}\\
&\ge\frac{\delta}{\Delta_{\mathcal Q,V}}\cdot\log\prod_{j=1}^{q}\frac{\|\tilde f(z)\|^d}{|D_j(\tilde f(z))|}+O(1).
\end{split}
\end{align}
Then, from (\ref{7}) and (\ref{8}) we have
\begin{align}\label{9}
\begin{split}
\frac{1}{\Delta_{\mathcal Q,V}(k+1)}&\cdot\log\prod_{j=1}^{q}\frac{\|\tilde f(z)\|^d}{|D_j(\tilde f(z))|}\\
&\le\dfrac{1}{u(n_u+1)}\left (\max_{\mathcal J\subset\mathcal L}\log\prod_{L\in\mathcal J}\dfrac{\|\tilde F(z)\|\cdot \|L\|}{|L(\tilde F(z))|}-(n_u+1)\log \|\tilde F(z)\|\right )\\
&+d\log \|\tilde f(z)\|+\frac{(2k+1)\delta}{u}\sum_{1\le j\le q}\log\frac{\|\tilde f(z)\|^d}{|D_j(\tilde f(z))|}+O(1),
\end{split}
\end{align}
where the term $O(1)$ does not depend on $z$.

By applying Theorem \ref{3.3} to the linear nondegenerate holomorphic curve $F$ and the system of linear forms $\mathcal L$, we get:
\begin{align}\label{10}
\frac{1}{2\pi}\int_0^{2\pi}\max_{\mathcal J\subset\mathcal L} \log\prod_{L\in\mathcal J}\left(\frac{\|\tilde f\|}{|L(\tilde f)|}(re^{i\theta})\right)d\theta
 \leq (n_u+1)T_0(r,F)-N_0(r,\nu_{W(\tilde F)})+S_f(r),
\end{align}
where the maximum is taken over all subsets $\mathcal J$ of $\mathcal L$, such that $\sharp J=n_u+1$ and $\{L|L\in\mathcal J\}$ is linearly independent.

Integrating both sides of (\ref{9}), in the view of (\ref{10}), we obtain
\begin{align}\label{11}
\begin{split}
\frac{1}{\Delta_{\mathcal Q,V}}&\left(dqT_0(r,f)-\sum_{i=1}^qN_{Q_i(f)}(r)\right) \le d(k+1)T_0(r,f)-\frac{(k+1)}{u(n_u+1)}N_0(r,\nu_{W(\tilde F)})\\
&+\frac{(2k+1)(k+1)\delta}{u}\sum_{i=1}^qm_0(r,f,Q_i)+S_f(r).
\end{split}
\end{align}

We now estimate $N_0(r,\nu_{W(\tilde F)})$. Fix a point $z\in\C$. We set $c_{i}=\max\{0,\nu_{(Q_i,f)}(z)-n_u\}$ for $i=1,\ldots,q$, and
$${\bf c}=(c_{1},\ldots,c_{q})\in\mathbb Z^q_{\ge 0}.$$
Then, for $i=0,\ldots,n_u$, there are
$${\bf a}_i=(a_{i1},\ldots,a_{iq}),a_{is}\in\{1, \ldots ,u\}$$
such that ${\bf y}^{{\bf a}_0}, \ldots ,{\bf y}^{{\bf a}_{n_u}}$ is a basic of $I_u(Y)$ and
$$ S_{Y_z}(u,{\bf c})=\sum_{i=0}^{n_u}{\bf a}_i\cdot {\bf c}.$$
Similarly as above, ${\bf y}^{{\bf a}_i}=L_i(v_0,\ldots,v_{n_u})$, where $L_0,\ldots,L_{n_u}$ are linearly independent linear forms with coefficients in $\C$. By the property of the Wronskian, we have
$$W(\tilde F)=c\det\bigl (L_0(\tilde\Phi(\tilde f))^{(w)}),\ldots,L_{n_u}(\tilde\Phi(\tilde f))^{(w)}\bigl )_{0\leq w\leq n_u},$$
where $c$ is a non-zero constant. This yields that
$$ \nu_{W(\tilde F)}(z)\ge \sum_{i=0}^{n_u}\min_{0\le w\le n_u}\nu_{L_i(\tilde f)^{(w)}}(z).$$
We easily see that
$$\nu_{L_i(\tilde f)^{(w)}}(z)=\sum_{j=1}^qa_{ij}\nu_{D_j(\tilde f)^{(w)}}(z)\ge\sum_{j=1}^qa_{ij}\max\{0,\nu_{D_j(\tilde f)}(z)-n_u\}.$$
Thus, we have
\begin{align}\label{12}
 \nu_{W(\tilde F)}(z)\ge \sum_{i=0}^{n_u}{{\bf a}_i}\cdot{\bf c}=S_Y(u,{\bf c}).
\end{align}
Take an index $i_0$ such that $c_{\sigma_{i_0}(0)}\ge c_{\sigma_{i_0}(1)}\ge\cdots\ge c_{\sigma_{i_0}(q-1)}$. Hence, $c_{\sigma_{i_0}(j)}=0$ for all $j\ge l_{i_0}$. Then by Lemma \ref{4.2} we have
$$\Delta_{\mathcal Q,V}e_{Y}({\bf c})\ge\delta(c_{\sigma_{i_0}(0)}+\cdots +c_{\sigma_{i_0}(l_{i_0})})=\delta\sum_{j=1}^{q}\max\{0,\nu_{D_j(\tilde f)}(z)-n_u\}.$$
On the other hand, by Theorem \ref{4.1} we have that
\begin{align*}
\frac{1}{u(n_u+1)}S_{Y}(u,{\bf c}) &\ge\frac{1}{(k+1)\delta}e_Y({\bf c})-\frac{(2k+1)\delta}{u}\max_{1\le i\le q}c_{i}\\
&\ge \left(\frac{1}{\Delta_{\mathcal Q,V}(k+1)}-\frac{(2k+1)\delta}{u}\right)\sum_{j=1}^{q}\max\{0,\nu_{D_j(\tilde f)}(z)-n_u\}.
\end{align*}
Combining this inequality and (\ref{12}), we have
\begin{align*}
\dfrac{(k+1)}{u(n_u+1)}&\nu^0_{W(\tilde F)}(z)\ge\dfrac{(k+1)}{u(n_u+1)}S_Y(u,{\bf c})\\
&\ge\left(\frac{1}{\Delta_{\mathcal Q,V}}-\frac{(2k+1)(k+1)\delta}{u}\right)\sum_{j=1}^{q}\max\{0,\nu_{D_j(\tilde f)}(z)-n_u\}.
\end{align*}
Integrating both sides of this inequality, we obtain
$$\dfrac{(k+1)}{u(n_u+1)}N_0(r,\nu_{W(\tilde F)})\ge\left(\frac{1}{\Delta_{\mathcal Q,V}}-\frac{(2k+1)(k+1)\delta}{u}\right)\sum_{j=1}^{q}(N_0(r,\nu_{(Q_j,f)})-N_0^{[n_u]}(r,\nu_{(Q_j,f)})).$$
Combining (\ref{11}) and the above inequality and the first main theorem, we get
\begin{align*}
\frac{1}{\Delta_{\mathcal Q,V}}&\left(dqT_0(r,f)-\sum_{i=1}^qN_0(r,\nu_{(Q_i,f)})\right)\le d(k+1)T_0(r,f)\\
&-\left(\frac{1}{\Delta_{\mathcal Q,V}}-\frac{(2k+1)(k+1)\delta}{u}\right)\sum_{j=1}^{q}(N_0(r,\nu_{(Q_j,f)})-N_0^{[n_u]}(r,\nu_{(Q_j,f)}))\\
&+\frac{(2k+1)(k+1)\delta}{u}\sum_{i=1}^q(dT_0(r,f)-N_0(r,\nu_{(Q_i,f)}))+S_f(r).
\end{align*}
By setting $m_0=\frac{1}{\Delta_{\mathcal Q,V}}-\frac{(2k+1)(k+1)\delta}{u}$, the above inequality implies that
\begin{align*}
\left(q-\frac{k+1}{m_0}\right)T_0(r,f)\le\sum_{j=1}^{q}\dfrac{1}{d}N^{[n_u]}_0(r,\nu_{(Q_j,f)})+S_f(r).
\end{align*}
We choose $u=L(d,k,\deg V, \Delta_{\mathcal Q,V},\epsilon)= \lceil \Delta_{\mathcal Q,V}(2k+1)(k+1)d^k\deg V((k+1)\Delta_{\mathcal Q,V}+\epsilon)\epsilon^{-1}\rceil$. Then we have
$$ u\ge \Delta_{\mathcal Q,V}(2k+1)(k+1)\delta((k+1)\Delta_{\mathcal Q,V}\epsilon^{-1}+1) $$
and hence:
\begin{align*}
\frac{k+1}{m_0}&\le (k+1)\Delta_{\mathcal Q,V}+\epsilon;\\
n_u&\le H_{Y}(u)-1\le d^k\deg V\binom{u+k}{k}-1.
\end{align*}
Note that, we may suppose that $\epsilon<q-\Delta_{\mathcal Q,V}(k+1)$. Hence, if $k=1$ then
\begin{align*}
n_u+1&<d^n\deg V(1+u)\\
&<d^k\deg V(\Delta_{\mathcal Q,V}(2k+1)(k+1)d^k\deg V((k+1)\Delta_{\mathcal Q,V}\epsilon^{-1}+1)+2)\\
&\le d^{k^2+k}(\deg V)^{k+1}e^k\Delta_{\mathcal Q,V}^k(2k+5)^k((k+1)\Delta_{\mathcal Q,V}\epsilon^{-1}+1)^k.
\end{align*}
Otherwise, if $k\ge 2$, we have
\begin{align*}
n_u+1&<d^k\deg Ve^k\left(1+\frac{u}{k}\right)^k\\
&\le d^k\deg Ve^k\left(1+\frac{\Delta_{\mathcal Q,V}(2k+1)(k+1)d^k\deg V((k+1)\Delta_{\mathcal Q,V}\epsilon^{-1}+1)+1}{k}\right)^k\\
&\le d^{k^2+k}(\deg V)^{k+1}e^k\Delta_{\mathcal Q,V}^k(2k+5)^k((k+1)\Delta_{\mathcal Q,V}\epsilon^{-1}+1)^k.
\end{align*}
Therefore, we always have
$$n_u\le\left\lfloor d^{k^2+k}(\deg V)^{k+1}e^k\Delta_{\mathcal Q,V}^k(2k+5)^k((k+1)\Delta_{\mathcal Q,V}\epsilon^{-1}+1)^n\right\rfloor-1=L_0-1.$$
Then, we get
\begin{align*}
(q-\Delta_{\mathcal Q,V}(n+1)-\epsilon)T_0(r,f)\le\sum_{j=1}^{q}\frac{1}{d}N^{[L_0-1]}_0(r,\nu_{(Q_j,f)})+S_f(r).
\end{align*}
This completes the proof of the theorem.
\end{proof}

\section{Estimate the distributive constants of families of hypersurfaces in subgeneral position}

Firstly, we have the following lemma due to G. Dethloff and S. D. Quang \cite{Q26}
\begin{lemma}[{cf. \cite[Lemma 3.2]{Q26}}]\label{4.1} 
Let $V$ be an $k$-dimension subvariety of $\P^n(\C)$. Let $\mathcal Q=\{Q_1,\ldots,Q_q\}$ be a family of $q$ hypersurfaces of $\P^n(\C)$ in $N$-subgeneral position with respect to $V$ and satisfying the weak B\'ezout property.
Then there is a subset $\Gamma$ of $\{1,\ldots,q\}$ with $\sharp\Gamma\ge q-N+k-3+\lceil \frac{N-k+3}{2}\rceil$ such that the family of hypersurfaces $\mathcal P=\{Q_i;i\in\Gamma\}$ satisfies $\Delta_{\mathcal P,V}\le\frac{N-k+2}{2}$.
\end{lemma}

Motivated by the idea of using Nochka diagram of G. Heier and A. Levin in \cite{HL}, we now prove the main lemma of this paper.

\begin{lemma}\label{4.2}
Let $V\subset\P^n(\C)$ be an $k$-dimensional complex projective variety. Let $\mathcal{Q} = \{Q_1, \ldots, Q_q\}$ be a family of hypersurfaces in $\P^n(\C)$ located in $N$-subgeneral position with respect to $V\ (N> k)$ and satisfying the B\'ezout property. Then there exists a subset $\Gamma \subset \{1,\ldots,q\}$ such that $\sharp\Gamma\ge q -\left\lceil\frac{(N-k)\tau}{\tau-1}\right\rceil +1$ and the family $\mathcal P= \{Q_i : i \in \Gamma\}$ satisfies $\Delta_{\mathcal P,V}\le \tau$, where 
$$ \tau=\begin{cases}
1 + \sqrt{\frac{N-k}{k+1}}&\text{ if }\ k<N< \frac{5k+1}{4},\\
1+\frac{2(N-k)}{k+1}&\text{ if }\ \frac{5k+1}{4}\le N.
\end{cases} $$
\end{lemma}

\begin{proof} 
For any subset $\Gamma \subset \{1,\ldots,q\}$, denote
$$Q_\Gamma = \bigcap_{i \in \Gamma}Q_i \cap V, \quad c(\Gamma) = \operatorname{codim}_V Q_\Gamma.$$
Note that
$$\Delta_{\mathcal Q,V}= \max_{\Gamma \ne \emptyset} \frac{\sharp\Gamma}{c(\Gamma)}.$$
If $\Delta_{\mathcal Q,V}\le \tau$, then the conclusion holds with $\Gamma = \{1,\ldots,q\}$.

Assume $\Delta_{\mathcal Q,V}> \tau$. Let
$$\mathcal{S} = \left\{ \Gamma \subset \{1,\ldots,q\} : \frac{\sharp\Gamma}{c(\Gamma)} > \tau \right\}.$$
Then $\mathcal{S}$ is nonempty. Choose $\Gamma_1 \in \mathcal{S}$ such that $\sharp\Gamma_1$ is maximal. Since $\sharp\Gamma_1\le (N-k)+c(\Gamma_1),$ we have
$$ \tau<\frac{\sharp\Gamma_1}{c(\Gamma_1)}\le \frac{(N-k)+c(\Gamma_1)}{c(\Gamma_1)}.$$
This implies that
$$c(\Gamma_1)<\frac{N-k}{\tau-1}=\begin{cases}
\sqrt{(N-k)(k+1)}&\text{ if }\ k<N< \frac{5k+1}{4},\\
\frac{k+1}{2}&\text{ if }\ \frac{5k+1}{4}\le N.
\end{cases} $$ 
Hence $c(\Gamma_1)<\frac{k+1}{2}$ and then $c(\Gamma_1)\le \left\lfloor\frac{k}{2}\right\rfloor$.

\medskip
\noindent \textbf{Claim.} For any $\Gamma_2 \in \mathcal{S}\setminus\{\Gamma_1\}$, we have $\Gamma_1 \cap \Gamma_2 \ne \emptyset$.
\medskip

\noindent \textit{Proof of Claim.}
Assume $\Gamma_1 \cap \Gamma_2 = \emptyset$. Similar as above, one has $c(\Gamma_2)<\frac{k+1}{2}$.  Using the B\'ezout property, we have
$$c(\Gamma_1 \cup \Gamma_2) \le c(\Gamma_1) + c(\Gamma_2)<k+1.$$
This yields that $\bigcap_{j\in \Gamma_1 \cup \Gamma_2}(Q_j\cap V)\ne\varnothing$.

Also, from
$$\frac{\sharp\Gamma_1}{c(\Gamma_1)} > \tau \quad\text{ and }\quad \frac{\sharp\Gamma_2}{c(\Gamma_2)} > \tau,$$
by using the B\'ezout property we obtain
$$\frac{\sharp\Gamma_1 + \sharp\Gamma_2}{c(\Gamma_1 \cup \Gamma_2)}\ge \frac{\sharp\Gamma_1+\sharp\Gamma_2}{c(\Gamma_1) + c(\Gamma_2)} > \tau.$$
Thus $\Gamma_1 \cup \Gamma_2 \in \mathcal{S}$, contradicting the maximality of $\sharp\Gamma^*$. 
The claim is proved.

\medskip
Set $t = \frac{\sharp\Gamma_1}{c(\Gamma_1)} > \tau > 1.$ Since $\mathcal{Q}$ is in $N$-subgeneral position with respect to $V$, we have
$$c(\Gamma_1) \ge\sharp\Gamma_1-(N-k).$$
Hence
$$\sharp\Gamma_1\le \frac{(N-k)t}{t-1}<\frac{(N-k)\tau}{\tau-1},$$
This implies that
$$\sharp\Gamma_1\le\left\lceil\frac{(N-k)\tau}{\tau-1}\right\rceil -1.$$

%

Finally, suppose there exists $\Gamma' \subset \Gamma$ such that
$$
\frac{|\Gamma'|}{c(\Gamma')} > \tau.
$$
Then $\Gamma' \in \mathcal{S}$ but $\Gamma' \cap \Gamma_1 = \emptyset$, contradicting the claim. 
Hence
$$\Delta_{\mathcal P,V}\le \tau.$$
This completes the proof.
\end{proof}

\section{Proofs of the second main theorems in various cases}

\noindent
\textbf{5.1. Holomorphic curves from $\C$ and annuli into projective varieties}

Repeating the proof of Theorem \ref{1.1} in the case of holomorphic curves on $\C$, we obtain the following theorem.
\begin{theorem}\label{1.2}
Let $V\subset\P^n(\C)$ be a smooth complex projective variety of dimension $k\ge 1$. Let $\mathcal Q=\{Q_1,\ldots,Q_q\}$ be a family of hypersurfaces of $\P^n(\C)$ with the distributive constant $\Delta_{\mathcal Q,V}$ with respect to $V$. Let $d_i=\deg Q_i\ (1\le i\le q)$ and $d=lcm(d_1,\ldots,d_q)$. Let $f$ be an algebraically nondegenerate holomorphic curve from $\C$ into $V$. Then, for every $\epsilon >0$, 
$$\biggl \|\ (q-\Delta_{\mathcal Q,V}(k+1)-\epsilon) T_f(r)\le\sum_{i=1}^q\frac{1}{d_i}N^{[L_0-1]}(r,\nu_{(Q_i,f)})+o(T_f(r)).$$
where $L_0=L(d,k,\deg V,\Delta_{\mathcal Q,V},\epsilon)$.
\end{theorem}

\begin{proof}[Proof of Theorem \ref{1.1new}]
By Lemma \ref{4.2},  there is a subset $\Gamma$ of $\{1,\ldots,q\}$ with $\sharp\Gamma\ge q -\left\lceil\frac{(N-k)\tau}{\tau-1}\right\rceil +1$ such that the family of hypersurfaces $\mathcal P=\{Q_i;i\in\Gamma\}$ satisfies $\Delta_{\mathcal P,V}\le \tau$. Applying Theorem \ref{1.2} for $f$ and the family $\mathcal P$, we obtain
$$\biggl \|\ (q -\left\lceil\frac{(N-k)\tau}{\tau-1}\right\rceil +1-\Delta_{\mathcal P,V}(k+1)-\epsilon)T_f(r)\le\sum_{i\in\Gamma_1}\frac{1}{d_i}N^{[L_1']}(r,\nu_{(Q_i,f)})+o(T_f(r)),$$
where $L_1'=L(d,k,\deg V,\Delta_{\mathcal P,V},\epsilon)$. Since $\Delta_{\mathcal P,V}\le \tau$, we have $L_1'\le L_1$. Then, the above inequality implies that
$$\biggl \|\ \left (q -\left\lceil\frac{(N-k)\tau}{\tau-1}\right\rceil +1-\tau(k+1)-\epsilon\right)T_0(r,f)\le\sum_{i=1}^q\frac{1}{d_i}N^{[L_1]}(r,\nu_{(Q_i,f)})+o(T_f(r)).$$
The theorem is proved.
\end{proof}

\begin{proof}[Proof of Theorem \ref{1.2new}]

(a)  By Lemma \ref{4.1},  there is a subset $\Gamma_0$ of $\{1,\ldots,q\}$ with $\sharp\Gamma_0\ge q-\left\lfloor \frac{N-k+3}{2}\right\rfloor$ such that the family of hypersurfaces $\mathcal P_0=\{Q_i;i\in\Gamma_0\}$ satisfies $\Delta_{\mathcal P_0,V}\le\frac{N-k+2}{2}=\tau_0$. Applying Theorem \ref{1.1} for $f$ and the family $\mathcal P_0$, we obtain
$$\biggl \|\ \left (q-\left\lfloor \frac{N-k+3}{2}\right\rfloor-\Delta_{\mathcal P_0,V}(k+1)-\epsilon\right)T_0(r,f)\le\sum_{i\in\Gamma_0}\frac{1}{\deg Q_i}N^{[L_0']}_0(r,\nu_{(Q_i,f)})+S_f(r),$$
where $L_0'=L(d,k,\deg V,\Delta_{\mathcal P_0,V},\epsilon)$. Since $\Delta_{\mathcal P_0,V}\le \tau_0$, we have $L_0'\le L_0$. Then, the above inequality implies that
$$\biggl \|\ \left (q-\left\lfloor\tau_0+\frac{1}{2}\right\rfloor- \tau_0(k+1)-\epsilon\right)T_0(r,f)\le\sum_{i=1}^q\frac{1}{d_i}N^{[L_0]}(r,\nu_{(Q_i,f)})+S_f(r).$$
We get the desired inequality.

(b) Repeating the argument in the proof of (a), with Lemma \ref{4.2} in place of Lemma \ref{4.1}, yields the desired inequality. We omit the details of the proof.
\end{proof}

\noindent
\textbf{5.2. Holomorphic curves from complex discs into projective varieties.}

We have the following theorem due to the first author \cite{Q25}.

\vskip0.2cm
\noindent
\textbf{Theorem G} (reformulation, \cite{Q25}) {\it  Let $f$ be a algebraically nondegenerate holomorphic map of $\Delta (R)$ into an $k$-dimension smooth projective subvariety $V\subset\P^n(\mathbf{C})$ with finite growth index $c_f$. Let $\mathcal Q=\{Q_i\}_{i=1}^q$ be a family of hypersurfaces in $\P^n(\mathbf{C}$ with $\deg Q_i=d_i\ (1\le i\le q)$, and let $d=lcm(d_1,\ldots,d_q)$. Then for any $\epsilon>0$ and $\epsilon'>0$ we have 
\begin{align*}
\bigl\|\ &(q-\Delta_{\mathcal Q,V}(k+1)-\epsilon)T_f(r)\\
&\le\sum_{j=1}^{q}\frac{1}{d_j}N^{[L-1]}_{Q_j(f)}(r)+\frac{(\Delta_{\mathcal Q,V}(k+1)+\epsilon)(c_f+\epsilon')(L_0-1)}{2du_0}T_f (r),
\end{align*}
where $L_0=L(d,k,\deg V,\Delta_{\mathcal Q,V},\epsilon)$ and $u_0=u(d,k,\deg V,\Delta_{\mathcal Q,V},\epsilon)$.}

\begin{proof}[Proof of Theorem \ref{1.3new}]
We follow the argument in the proof of Theorem \ref{1.2new}, replacing Theorem \ref{1.1} by Theorem G.

(a)  Take the subset $\Gamma_0$ of $\{1,\ldots,q\}$ and the subfamily $\mathcal P_0=\{Q_i;i\in\Gamma_0\}$ of $\mathcal Q$ as the conslusion of  Lemma \ref{4.1}. Applying Theorem F for $f$ and the family $\mathcal P_0$, we obtain
\begin{align*}
\bigl\|\ &\left(q-\left\lfloor \frac{N-k+3}{2}\right\rfloor-\Delta_{\mathcal P_0,V}(k+1)-\epsilon\right)T_f(r)\\
&\le\sum_{j\in\Gamma_0}\frac{1}{d_j}N^{[L_0'-1]}_{Q_j(f)}(r)+\frac{(\Delta_{\mathcal P_0,V}(k+1)+\epsilon)(c_f+\epsilon')(L_0'-1)}{2du_0'}T_f (r),
\end{align*}
where $L_0'=L(d,k,\deg V,\Delta_{\mathcal P_0,V},\epsilon)$ and $u_0'=u(d,k,\deg V,\Delta_{\mathcal P_0,V},\epsilon)$. Since $\Delta_{\mathcal P_0,V}\le \tau_0$, we have $L_0'\le L_0$ and $\frac{L_0'-1}{u_0'}\le \frac{L_0-1}{u_0}$. Then, the above inequality implies that
\begin{align*}
\bigl\|\ &\left(q-\left\lfloor \frac{N-k+3}{2}\right\rfloor-\Delta_{\mathcal P_0,V}(k+1)-\epsilon\right)T_f(r)\\
&\le\sum_{j\in\Gamma_0}\frac{1}{d_j}N^{[L_0-1]}_{Q_j(f)}(r)+\frac{(\Delta_{\mathcal P_0,V}(k+1)+\epsilon)(c_f+\epsilon')(L_0-1)}{2du_0}T_f (r),
\end{align*}
We get the desired inequality.

(b) Repeating the argument in the proof of (a), with Lemma \ref{4.2} in place of Lemma \ref{4.1}, we will get the assertion (b).
\end{proof}

\noindent
\textbf{5.3. Holomorphic curves from K\"{a}hler manifolds into projective varieties.}

We have the following theorem due to T. D. Ngoc and S. D. Quang \cite{NQ} with minor correction.

\begin{theorem}[{correction, \cite[Theorem 1.1]{NQ}}] \label{TheoremG}
Let $M$ be an $m$-dimensional complete K\"ahler manifold  and $\omega$ be a K\"ahler form of $M.$  Assume that the universal covering of $M$ is biholomorphic to a ball in $\mathbb C^m.$ Let $f$ be an algebraically nondegenerate meromorphic map of $M$ into a subvariety $V$ of dimension $k$ in $\P^n(\C)$. Let $\mathcal Q=\{Q_1,\ldots,Q_q\}$ be a family of $q$ hypersurfaces in $\P^n(\C)$ with the distributive constant $\Delta_{\mathcal Q,V}$ with respect to $V$.  Let $d = lcm\{\deg D_1,\ldots,\deg D_q\}$. Assume that for some $\rho \ge 0,$ there exists a bounded continuous function $h \geq 0$ on $M$ such that
$$\rho\Omega_f + \mathrm{dd^c}\log h^2 \geq \mathrm{ric}\ \omega.$$
Then, for each  $\epsilon>0,$ we have
$$\sum_{j=1}^q {\delta}^{[L-1]}_{f} (Q_j) \le \Delta_{\mathcal Q,V}(k+1)+ \epsilon +\dfrac{\rho(L-1)(\Delta_{\mathcal Q,V}(k+1)+\epsilon)}{ud},$$
where $L=L(d,k,\deg V,\Delta_{\mathcal Q,V},\epsilon)$ and $u=u(d,k,\deg V,\Delta_{\mathcal Q,V},\epsilon).$
\end{theorem}

We give some sentences to explain the correction. Indeed, with the notation in Theorem \ref{TheoremG}, the defect relation obtained in \cite[Theorem 1.1]{NQ} is
$$\sum_{j=1}^q {\delta}^{[L-1]}_{f} (Q_j) \le \Delta_{\mathcal Q,V}(k+1)+ \epsilon +\dfrac{\rho(L-1)(k+1)}{ud}.$$
However, in the assumption 
$$\sum_{j=1}^q\delta^{[L-1]}_f(Q_j)>\dfrac{k+1}{x}+\dfrac{\rho (L-1)Lb}{d}$$ 
(see the third line from the bottom of page 692 in \cite{NQ}), the coefficient $x$ in the denominator of the second term was omitted. The correct assumption should be
$$\sum_{j=1}^q\delta^{[L-1]}_f(Q_j)>\frac{k+1}{x}+\dfrac{\rho (L-1)Lb}{xd}.$$
Here $x$ is a positive number satisfying
$$\frac{k+1}{x}\le \Delta_{\mathcal Q,V}(k+1)+\epsilon$$ 
(see \cite[inequality (13)]{NQ}), and  
$$b=\frac{k+1}{uL}.$$ 
Since the above assumption does not hold, we instead obtain
$$ \sum_{j=1}^q\delta^{[L-1]}_f(Q_j)\le \frac{k+1}{x}+\dfrac{\rho (L-1)Lb}{xd}$$
and consequently
$$\sum_{j=1}^q\delta^{[L-1]}_f(Q_j)\le \Delta_{\mathcal Q,V}(k+1)+\epsilon+ \dfrac{\rho(L-1)(\Delta_{\mathcal Q,V}(k+1)+\epsilon)}{ud}.$$

\begin{proof}[Proof of Theorem \ref{1.4new}]
We again use the argument in the proof of Theorem \ref{1.2new}.

(a)  Take the subset $\Gamma_0$ of $\{1,\ldots,q\}$ and the subfamily $\mathcal P_0=\{Q_i;i\in\Gamma_0\}$ of $\mathcal Q$ as the conclusion of  Lemma \ref{4.1}. Set $L_0'=L(d,k,\deg V,\Delta_{\mathcal P_0,V},\epsilon)$ and $u_0'=u(d,k,\deg V,\Delta_{\mathcal P_0,V},\epsilon)$. Since $\Delta_{\mathcal P_0,V}\le \tau_0$, we have $L_0'\le L_0$ and $\frac{L_0'-1}{u_0'}\le \frac{L_0-1}{u_0}$. Then, by applying \ref{TheoremG} for $f$ and the family $\mathcal P_0$, we obtain
\begin{align*}
\sum_{j=1}^q {\delta}^{[L_0-1]}_{f} (Q_j)&\le (q-\sharp\Gamma_0)+\sum_{j\in\Gamma_0}{\delta}^{[L_0'-1]}_{f}(Q_j)\\
&\le \left\lfloor \frac{N-k+3}{2}\right\rfloor+\Delta_{\mathcal P_0,V}(k+1)+ \epsilon +\dfrac{\rho(L_0'-1)(\Delta_{\mathcal P_0,V}(k+1)+\epsilon)}{u_0'd}\\
&\le \left\lfloor \frac{N-k+3}{2}\right\rfloor+\tau_0(k+1)+ \epsilon +\dfrac{\rho(L_0-1)(\Delta_{\mathcal Q,V}(k+1)+\epsilon)}{u_0d}.
\end{align*}
We get the desired inequality.

(b) Repeating the argument in the proof of (a), with Lemma \ref{4.2} in place of Lemma \ref{4.1}, we will get the assertion (b).
\end{proof}

\section{Schmidt's subspace theorem}

Let $k$ be a number field. Denote by $M_k$ the set of places of $k$ and by $M_k^\infty$ the set of archimedean places. For each $v\in M_k$, fix a normalized absolute value $|\cdot|_v$ such that $|\cdot|_v=|\cdot|$ on $\mathbb Q$ if $v$ is archimedean, and $|p|_v=p^{-1}$ if $v$ is non-archimedean lying above a prime $p$. Let $k_v$ be the completion of $k$ at $v$, and set
$$
n_v:=[k_v:\mathbb Q_v]/[k:\mathbb Q].
$$
Define $\|x\|_v=|x|_v^{n_v}$. Then the product formula reads
$$
\prod_{v\in M_k}\|x\|_v=1,\quad x\in k^*.
$$

For ${\bf x}=(x_0,\ldots,x_N)\in k^{N+1}$, set
$$
\|{\bf x}\|_v:=\max\{\|x_0\|_v,\ldots,\|x_N\|_v\},\quad v\in M_k.
$$
The (absolute logarithmic) height of ${\bf x}=(x_0:\cdots:x_N)\in\P^N(k)$ is
$$
h({\bf x}):=\sum_{v\in M_k}\log \|{\bf x}\|_v.
$$

Let $Q=\sum_{I\in\mathcal T_d}a_I{\bf x}^I$ be a homogeneous polynomial of degree $d$ in $k[x_0,\ldots,x_N]$, where ${\bf x}^I=x_0^{i_0}\cdots x_N^{i_N}$ for $I=(i_0,\ldots,i_N)$. Define
$$
\|Q\|_v:=\max\{\|a_I\|_v:\ I\in\mathcal T_d\}.
$$
For each $v\in M_k$, the Weil function $\lambda_{Q,v}$ is defined by
$$
\lambda_{Q,v}({\bf x}):=\log\frac{\|{\bf x}\|_v^d\cdot \|Q\|_v}{\|Q({\bf x})\|_v},\quad {\bf x}\in\P^N(k)\setminus\{Q=0\}.
$$

The following theorem is due to the first author \cite{Q22}.
\vskip0.2cm
\noindent
\textbf{Theorem F} (reformulation, \cite{Q22}) {\it  Let $k$ be a number field, $S$ be a finite set of places of $k$ and let $V$ be an irreducible projective subvariety of $\P^N$ of dimension $n$ defined over $k$. Let $\mathcal Q=\{Q_1,\ldots, Q_q\}$ be a family of $q$ homogeneous polynomials of $\bar k[x_0,\ldots,x_N]$ with the distributive constant $\Delta_{\mathcal Q,V}$ with respect to $V(\bar k)$. Then for each $\epsilon >0$, 
$$\sum_{v\in S}\sum_{j=1}^q\dfrac{\lambda_{Q_j,v}({\bf x})}{\deg Q_j}\le (\Delta_{\mathcal Q,V}(n+1)+\epsilon)h({\bf x})$$
for all ${\bf x}\in\P^N(k)$ outside a union of closed proper subvarieties of $V$.}

\begin{proof}[Proof of Theorem \ref{1.5new}]

(a)  Take the subset $\Gamma_0$ of $\{1,\ldots,q\}$ and the subfamily $\mathcal P_0=\{Q_i;i\in\Gamma_0\}$ of $\mathcal Q$ as the conclusion of  Lemma \ref{4.1}. Then, by applying Theorem F for $f$ and the family $\mathcal P_0$, we obtain
\begin{align*}
\sum_{v\in S}\sum_{j=1}^q\dfrac{\lambda_{Q_j,v}({\bf x})}{\deg Q_j}&\le (q-\sharp\mathcal P_0)h({\bf x})+\sum_{v\in S}\sum_{j\in\Gamma_0}\dfrac{\lambda_{Q_j,v}({\bf x})}{\deg Q_j}\\
&\le \left\lfloor \frac{\ell-n+3}{2}\right\rfloor h({\bf x})+\left(\Delta_{\mathcal P_0,V}(n+1)+\epsilon\right)h({\bf x})\\
&\le \left(\left\lfloor \frac{\ell-n+3}{2}\right\rfloor+\frac{(\ell-n+2)(n+1)}{2}+\epsilon\right)h({\bf x})
\end{align*}
for all ${\bf x}\in\P^N(k)$ outside a union of closed proper subvarieties of $V$.
We get the desired inequality.

(b) The proof of the assertion (b) is the same of the proof of the assertion (a) with Lemma \ref{4.2} in place of Lemma \ref{4.1}. Hence, we will omit the details of its proof.
\end{proof}

\section*{Data availability}
Data sharing not applicable to this article as no datasets were generated or analyzed during the current study.

\section*{Disclosure statement}
No potential conflict of interest was reported by the author(s).

\end{document}